\documentclass[12pt]{article}
\usepackage{e-jc}

\usepackage{latexsym}
\usepackage{amsmath,amsfonts,amssymb,amsthm}

\numberwithin{equation}{section}

\newtheorem{thm}{Theorem}[section]

\newtheorem{prop}[thm]{Proposition}

\theoremstyle{remark}

\theoremstyle{definition}

\input epsf

\newcommand{\ore}{\overrightarrow{e}}
\def\coef#1{\left\langle#1\right\rangle}
\def\FC{\operatorname{FC}}
\def\fc{\operatorname{fc}}

\catcode`\@=11
\@namedef{subjclassname@2010}{%
  \textup{2010} Mathematics Subject Classification}
\catcode`\@=12

\date{\dateline{Jan 13, 2011}{Jul 12, 2011}{July 22, 2011}}

\begin{document}

\newbox\Adr
\setbox\Adr\vbox{\hsize150mm
\centerline{$^\dagger$ Universit\'e Grenoble I, CNRS UMR 5582, 
Institut Fourier,}
\centerline{100 rue de maths, BP 74, 
F-38402 St.\ Martin d'H\`eres Cedex, France.}
\centerline{\footnotesize WWW: 
\tt http://www-fourier.ujf-grenoble.fr/\lower0.5ex\hbox{\~{}}bacher}
\vskip18pt
\centerline{$^*$ Fakult\"at f\"ur Mathematik, Universit\"at Wien,}
\centerline{Nordbergstra\ss e 15, A-1090 Vienna, Austria.}
\centerline{\footnotesize WWW: \footnotesize\tt
http://www.mat.univie.ac.at/\lower0.5ex\hbox{\~{}}kratt} 
}

\title{Chromatic statistics for triangulations and Fu{\ss}--Catalan complexes}

\author%[R. Bacher and C. Krattenthaler]
{
%\normalsize
\sc R. Bacher$^\dagger$ and C. Krattenthaler\protect\footnote{%
Research partially supported by the Austrian
Science Foundation FWF, grants Z130-N13 and S9607-N13,
the latter in the framework of the National Research Network
``Analytic Combinatorics and Probabilistic Number Theory."\newline
}
\\[18pt]
\box\Adr
}

\maketitle

\begin{abstract}
We introduce Fu\ss--Catalan complexes as $d$-dimensional generalisations of
triangulations of a convex polygon. These complexes are used to
refine Catalan numbers and Fu{\ss}--Catalan numbers, by introducing
colour statistics for triangulations 
and Fu{\ss}--Catalan complexes. 
Our refinements consist in showing that the number of triangulations,
respectively of Fu{\ss}--Catalan complexes, with a given colour distribution
of its vertices is given by closed product formulae.
The crucial ingredient in the proof is the Lagrange--Good inversion 
formula.
\end{abstract}

\newpage

\section{Introduction}

\subsection{Catalan and Fu{\ss}--Catalan numbers}

The sequence $(C_n)_{n\ge0}$ of Catalan numbers
$$1,1,2,5,14,42,132,429,1430,4862,16796,58786,\dots,$$
see \cite[sequence A108]{OEIS}, defined by 
\begin{equation} \label{eq:Cat} 
C_n:=\frac {1} {n+1}\binom {2n}n=\frac {1} {n}\binom {2n}{n-1},
\end{equation}
is ubiquitous in enumerative combinatorics.
Exercise~6.19 in \cite{StanBI} contains a list of 
66~sequences of sets enumerated by Catalan numbers,
with many more in the addendum \cite{Stadd}.
In particular, there are $\frac {1} {n+1}\binom {2n}n$
triangulations of a convex polygon\footnote{As is common, when we
speak of a ``convex polygon," we always tacitly assume that all its
angles are less than 180 degrees.} with $n+2$ vertices
(see \cite[Ex.~6.19.a]{StanBI}).

Even many years before Catalan's paper \cite{CataAA}, Fu{\ss} \cite{FussAA}
enumerated the dissections of a convex $((d-1)n+2)$-gon into
$(d+1)$-gons (obviously, any such dissection will consist of $n$
$(d+1)$-gons) and found that there are
\begin{equation} \label{eq:FC} 
\frac {1} {n}\binom {dn}{n-1}
\end{equation}
of those. These numbers are now commonly known as {\it
Fu\ss--Catalan numbers} (cf.\ \cite[pp.~59--60]{ArmDAA}).

Dissections of convex polygons into $(d+1)$-gons have been studied frequently
in the literature (see \cite{PrSiAA} for a survey). Moreover,
they have been recently embedded 
into a reflection group framework
in a very non-obvious way by Fomin and Reading \cite{FoReAA},
thereby extending earlier work of Fomin and Zelevinsky \cite{FOZeAB}.
For further combinatorial occurrences of the Fu{\ss}--Catalan
numbers, the reader is referred to \cite[paragraph after
(8.9)]{FoReAA}.

In the present paper, we propose a combinatorial
interpretation of Fu\ss--Catalan numbers which, to the best of our
knowledge, has not been considered before.
Nevertheless it is, in some sense, 
perhaps a (geometrically) more natural generalisation of triangulations
of a convex polygon (even if more difficult to visualise). Namely, we
consider $d$-dimensional simplicial complexes on $n+d$ vertices
homeomorphic to a $d$-ball that
consist of $n$ maximal simplices all of dimension $d$, with the additional
property that all simplices of dimension up to $d-2$ lie in the boundary
of the complex. (See Section~\ref{sectFussCatalan} for the
precise definition.) We call these complexes {\it
Fu{\ss}--Catalan complexes}. It is not difficult to see (cf.\
Section~\ref{sec:Def}) that the number of these complexes is indeed
given by the Fu\ss--Catalan number \eqref{eq:FC}.

We hope to provide sufficient evidence here 
that Fu\ss--Catalan complexes are generalisations of
triangulations which are equally attractive as 
dissections of convex polygons by $(d+1)$-gons. 
Some elementary properties of Fu\ss--Catalan complexes are
listed in Section~\ref{sec:elem}. Our main results 
(Theorems~\ref{thmcoeffabc},
\ref{thmCxxy}, \ref{thmgeneralcase}, and \ref{thmspec}) present
refinements of the plain enumeration of triangulations and
Fu\ss--Catalan complexes arising from certain vertex-colourings
of triangulations and Fu\ss--Catalan complexes, respectively.
It seems that these are intrinsic to Fu\ss--Catalan complexes;
in particular, 
we are not aware of any natural analogues of these results
for polygon dissections (except for the case of triangulations).
%As remarked by an anonymous 
%referee, they are perhaps interesting objects deserving more attention (we are 
%unaware of anterior considerations of these complexes).
%In this paper, they serve mainly as a tool and we mention only a few easy 
%and elementary properties of Fu{\ss}--Catalan complexes.
%The number of these complexes is again given
%by the Fu{\ss}--Catalan number \eqref{eq:FC}.

\subsection{Coloured refinements: short outline of this paper}

To each triangulation,
respectively, more generally, Fu{\ss}--Catalan complex,
we shall associate a colouring of its vertices. In a certain sense,
this colouring measures whether or not a large number of triangles
(respectively maximal simplices) meets in single vertices. 
We show that the number of triangulations of a convex $(n+2)$-gon
(respectively of $d$-dimensional Fu{\ss}--Catalan complexes on $n+d$
vertices) with a fixed distribution of colours of its
vertices is given by closed formulae (see Theorems~\ref{thmcoeffabc},
\ref{thmCxxy}, \ref{thmgeneralcase}, and \ref{thmspec}),
thus refining the Catalan numbers \eqref{eq:Cat} (respectively
the Fu{\ss}--Catalan numbers \eqref{eq:FC}).

In order to give a clearer idea of what we have in mind, we
shall use the remainder of this introduction to define precisely
the colouring scheme for the case of triangulations, and we shall
present the corresponding refined enumeration results (see
Theorems~\ref{thmcoeffabc} and \ref{thmCxxy}). 
Subsequently, in Section~\ref{sectFussCatalan} we generalise this
setting by introducing $d$-dimensional Fu{\ss}--Catalan complexes
for arbitrary positive integers $d$. The corresponding enumeration
results generalising Theorems~\ref{thmcoeffabc} and \ref{thmCxxy} 
are presented in Theorems~\ref{thmgeneralcase} and \ref{thmspec}. 
%Section~\ref{sectproofd=2} is then devoted to the proof of
%Theorem~\ref{thmcoeffabc}.
Section~\ref{sectproofgeneral} is then devoted to the proof of
Theorem~\ref{thmgeneralcase}, thus also establishing Theorem~\ref{thmcoeffabc}.
%The proof of the general case, Theorem~\ref{thmgeneralcase}, is slightly 
%more involved, and it is given in Section~\ref{sectproofgeneral}.
Crucial in this proof is the Lagrange--Good inversion formula
\cite{GoodAA}. Finally,
Section~\ref{sectspec} is devoted to the proof of Theorem~\ref{thmspec},
and thus also of Theorem~\ref{thmCxxy}, which it generalises.

\subsection{$3$-Coloured triangulations}

In the sequel, $P_n$ stands for a convex 
polygon with $n$ vertices. 
Since we are only interested 
in the combinatorics of triangulations of $P_{n+2}$,
we can consider a unique polygon $P_{n+2}$ for each integer $n\geq 0$.
A triangulation of $P_{n+2}$ has exactly $n$ triangles.
We shall always use the Greek letter $\tau$ to denote triangulations.
We call a triangulation $\tau$ of $P_{n+2}$ 
{\it $3$-coloured\/} if the $n+2$ vertices of $P_{n+2}$ are coloured 
with $3$ colours in such a way that the three vertices of every triangle 
in $\tau$ have different colours. (Using a graph-theoretic term,
we call a colouring with the latter property a {\it proper}
colouring.)
An easy induction on $n$ shows the existence of such a colouring,
and that it is unique up to permutations of all three colours. 

A {\it rooted\/} polygon is, by definition, a (convex) polygon containing 
a marked oriented edge $\ore$, the ``root edge"
(borrowing terminology from the theory of combinatorial maps;
cf.\ \cite{TuttAA}) in its boundary. In the illustrations
in Figure~\ref{fig1},
the marked oriented edge is always indicated by an arrow.
We write $P_{n+2}^\to$ for a rooted polygon with $n+2$ vertices.
In the sequel, we omit a separate discussion
of the degenerate case $n=0$,
where the rooted ``polygon" $P_2^\to$ essentially only consists of
the marked oriented edge $\ore$. We agree once and for all
that there is one triangulation in this case.

For $n\geq 1$,
a triangulation $\tau$ of $P_{n+2}^\to$ has a unique triangle
$\Delta_*$ that contains the marked oriented edge $\ore$. We consider
this ``root triangle'' as a triangle
with totally ordered vertices $v_0<v_1<v_2$, where $\ore$ starts at
$v_1$ and ends at $v_2$.
The $n+2$ vertices of a triangulation $\tau$ of $P_{n+2}^\to$ 
can then be uniquely coloured with three colours 
$\{\text{{\tt a},{\tt b},{\tt c}}\}$ such that 
$\ore$ starts at a vertex of colour {\tt b}, ends at a vertex
of colour {\tt c}, and vertices of every 
triangle $\Delta\in\tau$ have
different colours. 
Figure~\ref{fig1} shows all such $3$-coloured triangulations
of $P_{n+2}^\to$ for $n=0,1,2,3$.

\begin{figure}[h]
\epsfysize=5cm
\centerline{\epsfbox{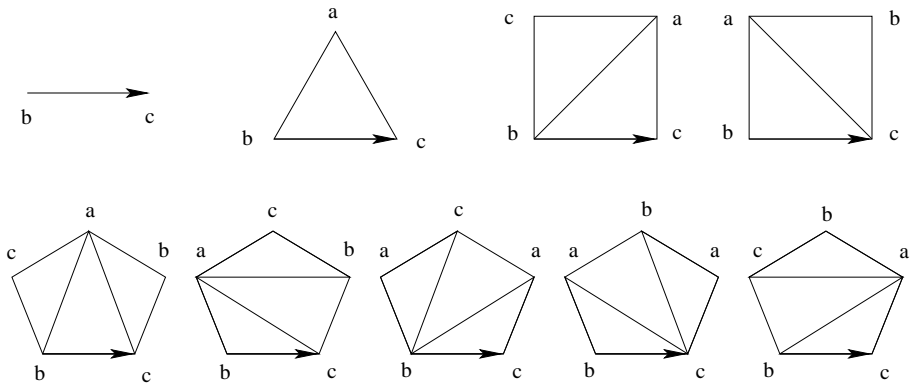}}
\caption{All $3$-coloured triangulations for $n=0,1,2,3$.}
\label{fig1}
\end{figure}

Our first result provides a closed formula for the
number of triangulations with a fixed colour distribution of its
vertices.

\begin{thm} \label{thmcoeffabc} 
Let $n$ be a non-negative integer and $\alpha, \beta, \gamma$
non-negative integers with $\alpha+\beta+\gamma=n+2$. Then
the number of triangulations of the rooted polygon $P_{n+2}^\to$
with $\alpha$ vertices of colour {\tt a},
$\beta$ vertices of colour {\tt b},
and $\gamma$ vertices of colour {\tt c} in the uniquely determined
colouring induced by a triangulation, in which
the starting vertex of the marked oriented edge $\ore$ has
colour {\tt b}, its ending vertex has colour {\tt c}, and
the three vertices in each triangle have different colours, 
is equal to
%$$C=bc+\sum_{(\alpha,\beta,\gamma)\in\mathcal A}
%C_{\alpha,\beta,\gamma}\ a^\alpha b^\beta c^\gamma\ ,$$
%where 
%and where the coefficient $C_{\alpha,\beta,\gamma}\in\mathbb N$
%for $(\alpha,\beta,\gamma)\in\mathcal A$ 
%is given by the formula
%$$\frac{(\alpha+\beta+\gamma-2)\quad
%(\alpha+\beta-2)!\,(\alpha+\gamma-2)!\,
%(\beta+\gamma-2)!}{(\alpha+\beta-\gamma)!\,(\alpha-\beta+\gamma)!\,(
%-\alpha+\beta+\gamma-1)!\,(\alpha-1)!\,(\beta-1)!\,(\gamma-1)!}$$
%if
%$$
%\alpha,\beta,\gamma\geq 1,\ 
%\alpha+\beta\geq \gamma,\
%\alpha+\gamma\geq \beta,\
%\beta+\gamma\geq \alpha+1,$$
%or equivalently by
%\begin{equation}\label{formulad=2a}
%\frac{s\ (s-\alpha)!\,(s-\beta)!\,(s-\gamma)!}{(s+1-2\alpha)!\,(s+2-2\beta)!\,
%(s+2-2\gamma)!\,(\alpha-1)!\,(\beta-1)!\,(\gamma-1)!}
%\end{equation}
%or
\begin{equation}\label{formulad=2}
\frac {\alpha(\alpha+\beta+\gamma-2)} 
{(\beta+\gamma-1){(\alpha+\gamma-1)}{(\alpha+\beta-1)}}
{\binom {\beta+\gamma-1} { \alpha}}
{{\binom {\alpha+\gamma-1} {\beta-1}}}
{{\binom {\alpha+\beta-1} { \gamma-1}}}\ .
\end{equation}
In the case where $\alpha=0$, this has to be
interpreted as the limit $\alpha\to0$, that is, it is $1$ if 
$(\alpha,\beta,\gamma)=(0,1,1)$ and $0$ otherwise.
\end{thm}

As we already announced, we shall generalise this theorem in
Theorem~\ref{thmgeneralcase} from triangulations 
to simplicial complexes. Its proof
(given in Section~\ref{sectproofgeneral}) shows that the corresponding
generating function, that is, the series
$$
C=C(a,b,c)=\sum_{\alpha,\beta,\gamma\ge0}C_{\alpha,\beta,\gamma}
a^\alpha b^\beta c^\gamma,
$$
where $C_{\alpha,\beta,\gamma}$ is the number of triangulations
in Theorem~\ref{thmcoeffabc},
is algebraic. To be precise, from the equations given in 
Section~\ref{sectproofgeneral} (specialised to $d=2$), one can extract that
\begin{equation} \label{eq:algGl} 
(bc)^3(1+a)+(bc)^2((b+c)a-1)C+(bc)^2(a-2)C^2+2bcC^3+bcC^4-C^5=0\ .
\end{equation}

\medskip
Next we identify two of the three colours. In other words,
we now consider {\it improper} colourings of triangulations of $P_{n+2}^\to$ 
by {\it two} colours, say black and white, 
such that every triangle has 
exactly one black vertex and two white vertices. 
There are then two possibilities to colour the marked oriented edge
$\ore$: either both of its incident vertices are coloured
white, or one is coloured white and the other black
(for the purpose of enumeration, it does not matter which of the 
two is white respectively black in the latter case). 
Remarkably, in both cases there exist again closed enumeration
formulae for the number of triangulations with a given colour
distribution.

\begin{thm} \label{thmCxxy} 
Let $n,b,w$ be non-negative integers with $b+w=n+2$.

\smallskip
{\em(i)} 
%The specialisation
%$Y=C(x,y,y)$ satisfies the equation
%$$(1+x)y^4-y^2(1+2y)Y+y(2+y)Y^2-Y^3=0$$
%and we have 
%$$C(x,y,y)=y^2+\sum_{n=0}^\infty \sum_{k=1}^{n+1}
%\frac {1} {k}{\binom {n+2k-1} { 2k-1}}
%{\binom {n} { k-1}}{x^ky^{n+2}}\ .$$
%(other version) 
%$$C(x,y,y)=y^2+\sum_{n=2}^\infty \sum_{k=1}^{n+1}
%\frac {1} {k}{\binom {n+2k-3} { 2k-1}}
%{\binom {n-2} { k-1}}{x^ky^{n}}\ .$$
The number of triangulations of the rooted polygon $P_{n+2}^\to$
with $b$ black vertices and $w$ white vertices
in the uniquely determined
colouring induced by a triangulation, in which
both vertices of the marked oriented edge $\ore$ are
coloured white, and, in each triangle, exactly two of
the three vertices are coloured white,
is equal to
$$\frac {2b} {(w-1)(2b+w-2)}{\binom {2b+w-2} { w-2}}
{\binom {w-1} { b}}\ .$$

\smallskip
{\em(ii)} 
%The specialisation
%$Z=C(x,x,y)$ satisfies the equation
%$$x^2y^2+xy(x-1)Z+Z^3=0$$
%and we have
%$$C(x,x,y)=xy+\sum_{n=1}^\infty\ \sum_{k=0}^n\frac {1} {n}
%{\binom {n+2k} { n-1}}{\binom {n} { k}}{x^{n+1}y^{k+1}}\ .$$
%(other version)
%$$C(x,x,y)=xy+\sum_{n=2}^\infty\ \sum_{k=1}^n\frac {1} {n-1}
%{\binom {n+2k-3} { n-2}}{\binom {n-1} { k-1}}{x^{n}y^{k}}\ .$$
The number of triangulations of the rooted polygon $P_{n+2}^\to$
with $b$ black vertices and $w$ white vertices 
in the uniquely determined
colouring induced by a triangulation, in which
the starting vertex of the marked oriented edge $\ore$ is
coloured white, its ending vertex is coloured black,
and, in each triangle, exactly two of
the three vertices are coloured white,
is equal to
$$
\frac {1} {2b+w-2}
{\binom {2b+w-2} { w-1}}{\binom {w-1} { b-1}}\ .$$
\end{thm}

Obviously, the generating functions corresponding to the numbers
in the above theorem must be algebraic. To be precise, it follows
from \eqref{eq:algGl} that the series $Y=C(x,y,y)$
(the generating function for the numbers in item~(i) of 
Theorem~\ref{thmCxxy}) and the series
$Z=C(x,x,y)$ (the generating function for the numbers in item~(ii) of 
Theorem~\ref{thmCxxy}) satisfy the algebraic equations
\begin{equation} \label{eq:xyy} 
(1+x)y^4-y^2(1+2y)Y+y(2+y)Y^2-Y^3=0
\end{equation}
and
\begin{equation} \label{eq:xxy} 
x^2y^2+xy(x-1)Z+Z^3=0\ ,
\end{equation}
respectively. As we announced, Theorem~\ref{thmCxxy} will be
generalised from triangulations to simplicial complexes in
Theorem~\ref{thmspec}.

%The series $C=C(a,b,c)$ has of course the easy
%specialisation
%$$C(x,x,x)=x\frac{1-\sqrt{1-4x}}{2}=\sum_{n=0}^\infty 
%\frac {1} {n+1}{\binom {2n} { n}}
%{x^{n+2}}\ .$$
%This follows for example from the fact that the 
%algebraic equation for $C$ factorises as
%$$(x+C)^2(x+x^2-C)(x^3-xC+C^2)=0$$
%if we set $a=b=c=x$.

Clearly, if we identify all three colours, then we are back to counting
all triangulations of the polygon $P_{n+2}$, of which there are
$C_n=\frac {1} {n+1}\binom {2n}n$.

\medskip
We end this introduction by mentioning that 
checkerboard colourings
of triangulations (obtained by colouring adjacent triangles with 
different colours chosen in a set of two colours) encode winding
properties of the corresponding $3$-vertex colouring.
Indeed, a $3$-coloured triangulation $\tau$ of $P_{n+2}$ 
induces a unique piecewise affine map $\varphi$ from $P_{n+2}$ onto 
a vertex-coloured triangle $\Delta$ such that $\varphi$ is colour-preserving
on vertices and induces affine bijections between triangles 
of $\tau$ and $\Delta$.
The map $\varphi$ is orientation-preserving, respectively 
orientation reverting, on black, respectively white, triangles
of $\tau$ endowed with a suitable black-white checkerboard colouring.
Restricting $\varphi$ to the oriented boundary of $P_{n+2}$ we get
a closed oriented path contained in the boundary 
of $\Delta$. 
The winding number of this path with respect to an interior point 
of $\Delta$ is given by the difference of black and white triangles
in the checkerboard colouring mentioned above. 
%This construction has an obvious generalization to Fu{\ss}-Catalan
%complexes. 
The resulting statistics 
for Catalan numbers (and the obvious generalization to
Fu{\ss}--Catalan numbers obtained by replacing winding numbers 
with the corresponding homology classes) have been studied by Callan 
in \cite{C}.

\section{Refinements of Fu{\ss}--Catalan numbers}\label{sectFussCatalan}

\subsection{Fu{\ss}--Catalan complexes}
\label{sec:Def}

Given an integer $d\geq 2$, we define
a {\it $d$-dimensional 
Fu{\ss}--Catalan complex of index $n\ge1$} to be a 
simplicial complex $\Sigma$ such that:

\begin{enumerate}
\item[(i)] $\Sigma$ is a $d$-dimensional simplicial complex
homeomorphic to a closed $d$-dimensional ball
having $n$ simplices of maximal dimension $d$.
\item[(ii)] All simplices of dimension up to $d-2$ of $\Sigma$ are contained 
in the boundary $\partial \Sigma$ (homeomorphic to a $(d-1)$-dimensional
sphere) of $\Sigma$. (Equivalently, 
the $(d-2)$-skeleton of $\Sigma$ is contained in its boundary 
$\partial\Sigma$).
\end{enumerate}

Such a complex $\Sigma$ is {\it rooted\/} if its boundary 
$\partial\Sigma$ contains a marked $(d-1)$-simplex, $\Delta_*$ say,
with totally ordered vertices. 
We denote a rooted $d$-dimensional
Fu{\ss}--Catalan complex by the pair $(\Sigma,\Delta_*)$.
By convention, 
a rooted $d$-dimensional Fu{\ss}--Catalan complex of index $0$ is 
given by $(\Delta_*,\Delta_*)$, where 
$\Delta_*$ is a simplex of dimension $d-1$ with totally ordered vertices. 

Rooted $d$-dimensional Fu{\ss}--Catalan complexes are generalisations of
rooted triangulations of polygons. In particular, a rooted
$2$-dimensional Fu{\ss}--Catalan complex of index $n$ is a 
triangulation of the rooted
polygon $P_{n+2}^\to$ with $n+2$ vertices.

Let $\fc_d(n)$ denote the number of $d$-dimensional rooted
Fu\ss--Catalan complexes of index $n$, and let $\FC_d(z)=\sum
_{n\ge 0} ^{}\fc_d(n)z^{n}$ be the corresponding generating
function. Consider a rooted $d$-dimensional 
Fu{\ss}--Catalan complex $(\Sigma,\Delta_*)$. The marked
$(d-1)$-dimensional simplex 
$\Delta_*$ is contained in a unique $d$-dimensional simplex of
$\Sigma$, which we
call the {\it root simplex} of the complex.
By deleting the root simplex,
we are left with a set of $d$ smaller Fu\ss--Catalan complexes ---
the $d$ subcomplexes which were ``glued'' to the $d$ facets of the
root simplex. (This is the extension of the standard decomposition of
a rooted triangulation when one removes the ``root triangle'').
These subcomplexes inherit also naturally a marked
$(d-1)$-dimensional simplex; that is, they are rooted Fu\ss--Catalan
complexes themselves. 
Namely, if $v_1<v_2<\dots<v_d$ is the
total order of the vertices of $\Delta_*$ and $v_0$ is the additional
vertex of the root simplex (containing $\Delta_*$), then we impose
the order
\begin{equation} \label{eq:order}
v_0<v_1<\dots<v_d
\end{equation}
on the vertices of the root simplex, and we declare the
$(d-1)$-dimensional simplex in which the subcomplex intersects the
root simplex to be the marked simplex of the subcomplex, together with
the total order which results from \eqref{eq:order} by restriction.
This decomposition leads directly to the functional equation
$$
\FC_d(z)=1+z\big(\FC_d(z)\big)^d.
$$
Under the substitution $\FC_d(z)=1+f_d(z)$, this is equivalent to
$$
\frac {f_d(z)} {\big(1+f_d(z)\big)^d}=z.
$$
This shows that $f_d(z)$ is the compositional inverse series of
$z/(1+z)^d$. Consequently, the coefficient of $z^n$ in $f_d(z)$, 
which equals the number $\fc_d(n)$, can be found using the Lagrange
inversion formula (cf.\ \cite[Theorem~5.4.2 with $k=1$]{StanBI}).
The result is the Fu\ss--Catalan number \eqref{eq:FC}; that is, the
number of $d$-dimensional rooted Fu\ss--Catalan complexes of index $n$
is indeed given by $\frac {1} {n}\binom {dn}{n-1}$.

\subsection{Elementary properties of Fu{\ss}--Catalan complexes}
\label{sec:elem}

A Fu{\ss}--Catalan complex is completely determined by its
$1$-skeleton. This is seen by gluing simplices onto all cliques 
(maximal complete subgraphs) of the
$1$-skeleton. The boundaries of two different 
rooted Fu{\ss}--Catalan complexes 
of dimension $>2$ are thus always 
combinatorially inequivalent when taking into account the marked 
simplex $\Delta_*$ with its totally ordered vertices.

However, the dimension $d$ and the number of vertices (or, 
equivalently, $d$ and the number of $d$-dimensional
simplices) determine the number of simplices of given dimension
in a Fu{\ss}--Catalan complex completely: for $i=0,1,\dots,d-1$, let 
$\tilde f_i(n,d)$ denote the number 
of $i$-simplices contained in the
boundary of a $d$-dimensional Fu{\ss}--Catalan complex $(\Sigma,\Delta_*)$
consisting of $n>0$ simplices of maximal dimension $d$.
(The interior of $\Sigma$ contains of course $n$ simplices of maximal dimension
$d$ separated by $(n-1)$ simplices of dimension $d-1$.)
We then have $\tilde f_i(1,d)=\binom {d+1} {i+1}$ for $i\in\{0,1,\dots,d-1\}$
since a Fu{\ss}--Catalan complex with $n=1$ is a $d$-dimensional
simplex, which has $\binom {d+1} {i+1}$ simplices of dimension $i$.
For $n\geq 1$, there hold the explicit formulae
\begin{align}
\tilde f_{d-1}(n,d)&=n(d-1)+2,\\
\tilde f_i(n,d)&=n{\binom d i}+\binom {d}{i+1},
\quad \quad \text {for }i=0,1,\dots,d-2. 
\end{align}
Indeed,
gluing an additional $d$-dimensional simplex to a Fu{\ss}--Catalan complex
adds one vertex and $d$ new $(d-1)$-dimensional simplices on the boundary and
hides a unique  $(d-1)$-dimensional simplex in the interior.
Moreover, for $i<d-1$, an $i$-simplex is either contained in the boundary of
the old complex or it involves the newly added point and is entirely 
contained in the added new $d$-dimensional simplex. 
In particular, there are $\binom di$ $i$-simplices of the latter kind.

It is natural to ask whether Fu\ss--Catalan complexes admit
``natural" realisations as polytopes. We shall present such
a realisation in the next paragraph. It is based on the observation
that Fu\ss--Catalan complexes of dimension $d$ can equivalently be
described by $(d-1)$-dimen\-sional {\it Schlegel diagrams}.
In order to explain this alternative description,
we embed a given Fu\ss--Catalan complex as a $d$-dimen\-sional convex
polytope $\mathcal P$ of $\mathbb R^d$. We choose now a point $O\in 
\mathbb R^d\setminus\mathcal P$ such that the convex hull of $\mathcal P$
and $O$ is obtained by gluing a unique simplex spanned by $O$ and 
$\Delta_*$ onto $\mathcal P$. 
We require moreover that every segment
joining $O$ to a vertex of $\mathcal P\setminus \Delta_*$ intersects 
the marked boundary simplex $\Delta_*$ in its interior. 
The central projection of $\mathcal P$ onto $\Delta_*$
with respect to the point $O$ is then called a Schlegel diagram
of $\mathcal P$. It contains all the combinatorial information 
allowing the reconstruction of the initial Fu\ss-Catalan complex.
More precisely, it is given (up to combinatorial
equivalence) by so-called barycentric subdivisions
starting with the marked 
simplex $\Delta_*$ (which, 
as always, we consider with the extra-structure given by its
completely ordered vertices): 
a barycentric subdivision of a $(d-1)$-dimensional
simplex $\Delta$ with vertices $\mathcal V$
is obtained by partitioning $\Delta$ into $d$ simplices $\Delta_v$,
indexed by $v\in\mathcal V$,
defined by considering the convex hull of $\mathcal V\setminus\{v\}$
and of the barycenter $b=\frac{1}{d}\sum_{w\in\mathcal V}w$ of 
$\Delta$. Iterating barycentric subdivisions $n$ times in all possible 
ways starting with 
the $(d-1)$-dimensional simplex $\Delta_*$ gives exactly the set of
all Schlegel diagrams (as described above) of all $d$-dimensional
Fu\ss--Catalan complexes consisting of $n$ simplices of maximal dimension
$d$. 
Note that barycentric subdivisions add only points with rational 
coordinates if all vertices of $\Delta_*$ have rational coordinates
(more precisely, all vertices belong to $A\left(\mathbb Z\left[\frac{1}{d}\right]
\right)^d$ if $A$ is a positive integer such that $A\Delta_*$ 
has integral coordinates). A pleasant feature of
barycentric subdivisions is the fact that they carry a natural
distributive lattice structure (defined by unions and intersections).

A (more or less) natural polytope $\mathcal P\subset \mathbb R^d$
representing a given $d$-dimensional Fu\ss-Catalan complex $(\Sigma,\Delta_*)$
can now be constructed as follows: choose the $d$ ordered points
$$(1-d,1,1,\dots,1)<(1,1-d,1,1,\dots,1)<\dots<(1,1,\dots,1,1,d-1)$$
of $\mathbb Z^d$
as vertices for $\Delta_*$ and use $\Delta_*$ for constructing 
a barycentric subdivision $BS$ corresponding to $(\Sigma,\Delta_*)$.
Associate to a  vertex $V=(a_1,a_2,\dots,a_d)$ of $BS$ the 
point 
$$\tilde V=(a_1,a_2,\dots,a_d)+(1,1,\dots,1)\sum_{j=1}^da_j^2\in \mathbb Q^d\ .$$
The set of all points $\tilde V$
associated to vertices of $BS$ is then the
set of vertices of a polytope realising $(\Sigma,\Delta_*)$ in 
$\mathbb R^d$.

\subsection{$(d+1)$-colourings of $d$-dimensional 
Fu{\ss}--Catalan complexes}
\label{subsectdcoulrings}

Let $\mathcal C$ be a set of colours.
A {\it proper colouring} of a simplicial complex $\Sigma$ with vertex set
$\mathcal V$ by colours from $\mathcal C$ 
is a map $\gamma:\mathcal V\longrightarrow
\mathcal C$  such that $\gamma(v)\ne\gamma(w)$
for any pair of vertices $v,w$ defining a $1$-simplex of $\Sigma$.
Equivalently, a proper colouring of a simplicial complex $\Sigma$ is 
a proper colouring of the graph defined by the $1$-skeleton of $\Sigma$.

Every rooted $d$-dimensional Fu{\ss}--Catalan complex $(\Sigma,\Delta_*)$
has a unique colouring by $(d+1)$ totally ordered
colours $c_0<c_1<\dots<c_d$ such that the $i$-th vertex of $\Delta_*$
(in the given total order of the vertices of $\Delta_*$)
has colour $c_{i}$, $i=1,2,\dots,d$.
The following theorem presents a closed formula for the number of
Fu{\ss}--Catalan complexes of index $n$ with a given colour
distribution.

\begin{thm}\label{thmgeneralcase} 
Let $d$, $n$, $\gamma_0,\gamma_1,\dots,\gamma_d$ 
be non-negative integers with $d\ge2$ and
$\gamma_0+\gamma_1+\dots+\gamma_d=n+d$. Then
the number of $d$-dimensional Fu{\ss}--Catalan complexes
$(\Sigma,\Delta_*)$ of index $n$
with $\gamma_i$ vertices of colour $c_i$, $i=0,1,\dots,d$,
in the uniquely determined proper colouring by the colours
$c_0,c_1,\dots,c_d$ in which
the $i$-th vertex of the root simplex $\Delta_*$ has colour $c_i$,
$i=1,2,\dots,d$,
is equal to
%the coefficients $C_{d,\gamma_0,\gamma_1,\dots,\gamma_{d}}$ of the series
%$$C_d(x_0,x_1,\dots,x_{d})=x_1x_2\cdots x_{d}+
%\sum_{\gamma_0,\gamma_1,\dots,\gamma_{d}\geq 1}
%C_{d,\gamma_0,\gamma_1,\dots,\gamma_{d}}\
%x_0^{\gamma_0}x_1^{\gamma_1}\cdots x_{d}^{\gamma_{d}}$$
%are given by
\begin{equation}\label{formdimd}
%&\frac{s^{d-1}\prod_{j=0}^{d}(s-\gamma_j)!}{(\gamma_0-1)!\,(s+1-2\gamma_0)!
%\prod_{j=1}^{d}\left((s+2-2\gamma_j)!\,(\gamma_j-1)!\right)}\\
%\label{formdimd2}
%&=
s^{d-1}\frac {\gamma_0} {s-\gamma_0+1}{\binom {s-\gamma_0+1} {
\gamma_0}}\prod_{j=1}^{d}\frac {1} {s-\gamma_j+1}
{{\binom {s-\gamma_j+1
} { \gamma_j-1}}}\ ,
\end{equation}
where $s=-d+\sum_{j=0}^{d}\gamma_j$.
%if the arguments of all involved factorials are $\geq 0$, and 
%$C_{d,\gamma_0,\gamma_1,\dots,\gamma_{d}}=0$ otherwise.
In the case where $\gamma_0=0$, this has to be
interpreted as the limit $\gamma_0\to0$, that is, it is $1$ if 
$(\gamma_0,\gamma_1,\dots,\gamma_d)=(0,1,1,\dots,1)$ and $0$ otherwise.
\end{thm}

Formula~\eqref{formdimd} generalises Formula~\eqref{formulad=2},
the latter corresponding to the case $d=2$ of the former.

\subsection{Specialisations obtained by identifying colours}

Generalising the scenario in Theorem~\ref{thmCxxy}, we now identify
some of the colours. Namely,
given a non-negative integer $k$ and $k+1$ positive integers 
$\beta_0,\beta_1,\beta_2,\dots,\beta_k$ with
$\beta_0+\beta_1+\beta_2+\dots+\beta_k=d+1$,
we set
\begin{align*}
c_0=\dots=c_{\beta_0-1}&=c'_0\\
c_{\beta_0}=\dots=c_{\beta_0+\beta_1-1}&=c'_1\\
&\vdots\\
c_{\beta_0+\beta_1+\dots+\beta_{i-1}}=\dots
=c_{\beta_0+\beta_1+\dots+\beta_i-1}&=c'_i\\
&\vdots\\
c_{\beta_0+\beta_1+\dots+\beta_{k-1}}=\dots
=c_{d}&=c'_k.\\
\end{align*}
Given a rooted Fu{\ss}--Catalan complex $(\Sigma,\Delta_*)$ 
with its uniquely determined
colouring as in Theorem~\ref{thmgeneralcase}, after this
identification we obtain a colouring of the simplices of
$(\Sigma,\Delta_*)$ in which each $d$-dimensional simplex
has $\beta_i$ vertices of colour $c'_i$, $i=0,1,\dots,k$. Our next theorem 
presents a closed formula for the number of $d$-dimensional 
Fu{\ss}--Catalan complexes of index $n$ with a given colour
distribution {\it after} this identification of colours.

\begin{thm}\label{thmspec} 
Let $d$, $k$, $n$, 
$\beta_0,\beta_1,\dots,\beta_k$, 
$\gamma_0,\gamma_1,\dots,\gamma_k$ 
be non-negative integers with $d\ge2$,
$\beta_0+\beta_1+\beta_2+\dots+\beta_k=d+1$, and
$\gamma_0+\gamma_1+\dots+\gamma_k=n+d$. Then
the number of $d$-dimensional Fu{\ss}--Catalan complexes
$(\Sigma,\Delta_*)$ of index $n$
with $\gamma_i$ vertices of colour $c'_i$, $i=0,1,\dots,k$,
in the uniquely determined colouring in which
the first $\beta_0-1$ vertices of the root simplex $\Delta_*$ have 
colour $c'_0$, 
the next $\beta_1$ vertices have colour $c'_1$, 
the next $\beta_2$ vertices have colour $c'_2$, \dots,
the last $\beta_k$ vertices have colour $c'_k$, 
and in which each $d$-dimensional simplex has $\beta_i$ vertices
of colour $c'_i$, $i=0,1,\dots,k$,
is equal to
%A non-zero coefficient $\sigma_{\gamma_0,\gamma_1,\dots,\gamma_{k}}$
%of $y_0^{\gamma_0}y_1^{\gamma_1}\cdots y_{k}^{\gamma_{k}}$ in the series
%$$S_{\beta_0,\beta_1,\beta_2,\dots,\beta_k}=y_0^{\beta_0}\prod_{j=1}^k
%y_j^{\beta_j}+\sum_{\gamma_0,\gamma_1,\dots,\gamma_d\ge1}
%\sigma_{\gamma_0,\gamma_1,\dots,\gamma_{k}}
%y_0^{\gamma_0}y_1^{\gamma_1}\cdots y_{k}^{\gamma_{k}}$$
%is given by the formula
\begin{equation*}
s^{k-1}\frac {\gamma_0-\beta_0+1} {\beta_0 s+\beta_0-\gamma_0}
{\binom {\beta_0 s+\beta_0-\gamma_0} { \gamma_0-\beta_0+1}}
\prod_{j=1}^k
\frac{\beta_j}{\beta_js+\beta_j-\gamma_j}
{\binom {\beta_js+\beta_j-\gamma_j} { 
\gamma_j-\beta_j}}\ ,
\end{equation*}
where $s=-d+\sum_{j=0}^{k}\gamma_j$.
\end{thm}

This theorem contains all the afore-mentioned results as special
cases. Clearly, Theorem~\ref{thmgeneralcase} is the special case
of Theorem~\ref{thmspec}
where $k=d$ and $\beta_0=\beta_1=\dots=\beta_d=1$
(and Theorem~\ref{thmcoeffabc} is the further special case in
which $d=2$). Item~(i) of Theorem~\ref{thmCxxy} results for
$d=2$, $k=1$, $\beta_0=1$, $\beta_1=2$, while item~(ii) results
for $d=2$, $k=1$, $\beta_0=2$, $\beta_1=1$. Moreover, upon
setting $k=0$ and $\beta_0=d+1$ in Theorem~\ref{thmspec}, we
obtain Formula~\eqref{eq:FC} (and \eqref{eq:Cat} in the further
special case where $d=2$).

\section{Generating functions and the Lagrange--Good inversion formula}
\label{sectproofgeneral}

In this section we provide the proof of Theorem~\ref{thmgeneralcase}. 
It makes use of generating function calculus, which serves to 
reach a situation in which the Lagrange--Good inversion formula
\cite{GoodAA}
(see also \cite[Sec.~5]{KratAC} and the references cited therein)
can be applied to compute the numbers that we are interested in.
The proof requires as well a determinant evaluation, which we
state and establish separately at the end of this section.

\begin{proof}[Proof of Theorem~\ref{thmgeneralcase}]
Let 
$$C_d(x_0,x_1,\dots,x_{d}):=
\sum_{(\Sigma,\Delta_*)}
x_0^{\gamma_0(\Sigma,\Delta_*)}
x_1^{\gamma_1(\Sigma,\Delta_*)}\cdots
x_d^{\gamma_d(\Sigma,\Delta_*)},
$$
where the sum is over all $d$-dimensional Fu{\ss}--Catalan complexes
$(\Sigma,\Delta_*)$ (of any index,
including the $(d-1)$-dimensional complex $(\Delta_*,\Delta_*)$
of index $0$), and where
$\gamma_i(\Sigma,\Delta_*)$ denotes the number of vertices of colour
$c_i$ in the unique colouring of $(\Sigma,\Delta_*)$ described
in the statement of Theorem~\ref{thmgeneralcase}.
It is our task to compute the coefficient of 
$x_0^{\gamma_0}x_1^{\gamma_1}\cdots x_d^{\gamma_d}$ in
the series $C_d(x_0,x_1,\dots,x_{d})$.

Starting from our generating function $C_d(x_0,x_1,\dots,x_{d})$,
we define $d+1$ series by cyclically permuting the variables,
\begin{align*}
C^{\{0\}}(x_0,x_1,\dots,x_d)&=C_d(x_0,x_1,\dots,x_{d}),\\
C^{\{1\}}(x_0,x_1,\dots,x_d)&=C_d(x_1,x_2,\dots,x_{d},x_0),\\
&\vdots\\
C^{\{d\}}(x_0,x_1,\dots,x_d)&=C_d(x_{d},x_0,x_1,\dots,x_{d-1}).
\end{align*} 
The decomposition of rooted $d$-dimensional 
Fu{\ss}--Catalan complexes
$(\Sigma,\Delta_*)$ determined by the unique $d$-dimensional simplex
containing $\Delta_*$, which we described in Section~\ref{sec:Def},
yields a system of equations relating these $d+1$ series.
To be precise, let $(\Sigma,\Delta_*)$ be
a rooted $d$-dimensional Fu{\ss}--Catalan complex of index $n\geq 1$,
and let $\Delta_*^d$ be its unique $d$-dimensional simplex 
containing $\Delta_*$. It intersects
$\Sigma\setminus \Delta_*^d$ along $d$ rooted
sub-Fu{\ss}--Catalan complexes,
with their marked $(d-1)$-dimensional simplices defined by 
their intersection with
the boundary of $\Delta_*^d$. These sub-complexes define a
decomposition of $(\Sigma,\Delta_*)$. It shows that
$$
C_d(x_0,x_1,\dots,x_d)
={x_1\cdots x_{d}}+\frac{1}{x_0(x_0x_1\cdots x_{d})^{d-2}}
{\prod_{j=1}^d}C^{\{j\}}(x_0,x_1,\dots,x_{d}),
$$
and, more generally,
\begin{multline}\label{eqndefCd}
C^{\{i\}}(x_0,x_1,\dots,x_d)
=\frac{x_0x_1\cdots x_{d}}{x_i}+\frac{1}{x_i(x_0x_1\cdots x_{d})^{d-2}}
\underset{j\ne i}{\prod_{j=0}^d}C^{\{j\}}(x_0,x_1,\dots,x_{d}),\\
i=0,1,\dots,d. 
\end{multline}
In order to simplify this system of equations,
we define $d+1$ series $g_0,g_1,\dots,g_{d}$ 
by the equations
\begin{equation}\label{defgfromC}
C^{\{i\}}(x_0,x_1,\dots,x_{d})=\frac{1}{x_i}
\left(1+g_i(x_0,x_1,\dots,x_{d})\right)
\prod_{j=0}^{d}x_j,\qquad  i=0,1,\dots,d.
\end{equation}
The reader should keep in mind that we want to compute the coefficient
of $x_0^{\gamma_0}x_1^{\gamma_1}\cdots x_d^{\gamma_d}$ in
the series $C_d(x_0,x_1,\dots,x_{d})$, that is, in terms of the new
series, the coefficient of
$x_0^{\gamma_0}x_1^{\gamma_1-1}x_2^{\gamma_2-1}\cdots x_d^{\gamma_d-1}$ in
the series $g_0(x_0,x_1,\dots,x_{d})$.

From now on, we suppress the arguments of series for the sake of
better readability; that is, we write $g_i$ instead of
$g_i(x_0,x_1,\dots,x_d)$, etc., for short. 
With this notation, the system \eqref{eqndefCd} becomes
$$g_i=\frac{x_i}{(1+g_i)}\prod_{j=0}^{d}(1+g_j),
\quad \quad i=0,1,\dots,d,$$
or, equivalently,
$$x_i=\frac{g_i(1+g_i)}{\prod_{j=0}^{d}(1+g_j)},\qquad
i=0,1,\dots,d.$$
By a straightforward application of the Lagrange--Good inversion
formula \cite{GoodAA},
we have
$$
\coef{\mathbf x^{\boldsymbol \gamma}}g_0
=\coef{\mathbf x^{-\mathbf1}}
x_0\det(J_{d+1})
\prod_{j=0}^{d}\frac{(1+x_j)^{d+\vert\boldsymbol\gamma\vert-\gamma_j}}{x_j^{\gamma_j+1}}\ ,$$
where $\coef{\mathbf x^{\boldsymbol \gamma}}g_0$ denotes the
coefficient of $x_0^{\gamma_0}x_1^{\gamma_1}\cdots x_d^{\gamma_d}$ in
the series $g_0$, $\coef{\mathbf x^{-\mathbf1}}f$ denotes the
coefficient of $x_0^{-1}x_1^{-1}\cdots x_d^{-1}$ in
the series $f$, 
$\vert\boldsymbol\gamma\vert$ stands for $\sum_{j=0}^{d}\gamma_j$, and 
$J_{d+1}$ is the Jacobian of the map $(x_0,x_1,\dots,x_{d})
\longmapsto (y_0,y_1,\dots,y_{d})$
defined by
$$y_i=\frac{x_i(1+x_i)}{\prod_{j=0}^{d}(1+x_j)},\quad \quad 
i=0,1,\dots,d.$$
A simple computation yields that
the entries of $J_{d+1}$ are given by 
\begin{alignat*}2
(J_{d+1})_{i,j}&
%\textstyle
=-\frac{x_i(1+x_i)}{(1+x_j)\prod_{k=0}^{d}(1+x_k)},&\qquad 
\text {if }i&\ne j,\\
(J_{d+1})_{i,i}&
%\textstyle
=\frac{1+x_i}{\prod_{k=0}^{d}(1+x_k)}.&\quad 
\end{alignat*}
By Proposition~\ref{propdetjac} at the end of this section, it follows that
\begin{align*}
\coef{\mathbf x^{\boldsymbol \gamma}}g_0
&=\coef{\mathbf x^{-\mathbf1}}
x_0\bigg(\prod_{j=0}^{d}\frac{(1+x_j)^{\vert\boldsymbol\gamma\vert-\gamma_j-1}
(1+2x_j)}
{x_j^{\gamma_j+1}}\bigg)
\bigg(1-\sum_{k=0}^{d}\frac {x_k} {1+2x_k}\bigg)\\
&=\coef{\mathbf x^{\boldsymbol \gamma}}
x_0\bigg(\prod_{j=0}^{d}(1+x_j)^{\vert\boldsymbol\gamma\vert-\gamma_j-1}
(1+2x_j)\bigg)
\bigg(1-\sum_{k=0}^{d}\frac {x_k} {1+2x_k}\bigg).
\end{align*}
Consequently, we get
\begin{align*}
&\kern-5pt
\coef{\mathbf x^{\boldsymbol \gamma}}g_0\\
&=\left({\binom {\vert\boldsymbol\gamma\vert-\gamma_0-1} 
{ \gamma_0-1}}+2{\binom {\vert\boldsymbol\gamma\vert-\gamma_0-1} {
\gamma_0-2}}
\right)\prod_{j=1}^{d}\left({\binom
  {\vert\boldsymbol\gamma\vert-\gamma_j-1} 
{ \gamma_j}}+2
{\binom {\vert\boldsymbol\gamma\vert-\gamma_j-1} { \gamma_j-1}}\right)\\
&\kern1cm
-{\binom {\vert\boldsymbol\gamma\vert-\gamma_0-1} { \gamma_0-2}}
\prod_{j=1}^{d}\left({\binom {\vert\boldsymbol\gamma\vert-\gamma_j-1} 
{ \gamma_j}}+2
{\binom {\vert\boldsymbol\gamma\vert-\gamma_j-1} { \gamma_j-1}}\right)\\
&\kern1cm
-\left({\binom {\vert\boldsymbol\gamma\vert-\gamma_0-1} 
{ \gamma_0-1}}+2{\binom {\vert\boldsymbol\gamma\vert-\gamma_0-1} {
\gamma_0-2}}
\right)\kern2cm\\
&\kern2cm
\times \sum_{k=1}^{d}{\binom
  {\vert\boldsymbol\gamma\vert-\gamma_k-1} 
{ \gamma_k-1}}\underset{j\ne k}{\prod_{j=1}^d}
\left({\binom {\vert\boldsymbol\gamma\vert-\gamma_j-1} { \gamma_j}}+2
{\binom {\vert\boldsymbol\gamma\vert-\gamma_j-1} { \gamma_j-1}}\right)\ .
\end{align*}
Setting 
\begin{align*}
P&=\prod_{j=1}^{d}\left({\binom
  {\vert\boldsymbol\gamma\vert-\gamma_j-1} 
{ \gamma_j}}+2
{\binom {\vert\boldsymbol\gamma\vert-\gamma_j-1} { \gamma_j-1}}\right)\\
&=\vert\boldsymbol\gamma\vert^d\prod_{j=1}^{d}
\frac{(\vert\boldsymbol\gamma\vert-\gamma_j-1)!}
{\gamma_j!\,(\vert\boldsymbol\gamma\vert-2\gamma_j)!}\ ,
\end{align*}
we can rewrite this as
\begin{align*}
\coef{\mathbf x^{\boldsymbol \gamma}}g_0&=
\left({\binom {\vert\boldsymbol\gamma\vert-\gamma_0-1} { \gamma_0-1}}
+{\binom {\vert\boldsymbol\gamma\vert-\gamma_0-1} {
\gamma_0-2}}
\right)P\\
&\kern1cm
-\left({\binom {\vert\boldsymbol\gamma\vert-\gamma_0-1} { \gamma_0-1}}
+2{\binom {\vert\boldsymbol\gamma\vert-\gamma_0-1} {
\gamma_0-2}}
\right)P
\sum_{k=1}^{d}\frac{{\binom {\vert\boldsymbol\gamma\vert-\gamma_k-1} 
{ \gamma_k-1}}}
{{\binom {\vert\boldsymbol\gamma\vert-\gamma_k-1} { \gamma_k}}+2
{\binom {\vert\boldsymbol\gamma\vert-\gamma_k-1} { \gamma_k-1}}}\\
&=
{\binom {\vert\boldsymbol\gamma\vert-\gamma_0} { \gamma_0-1}}P
-\frac{\vert\boldsymbol\gamma\vert-1}{\vert\boldsymbol\gamma\vert-\gamma_0}
{\binom {\vert\boldsymbol\gamma\vert-\gamma_0} { \gamma_0-1}}P
\sum_{k=1}^{d}\frac{\gamma_k}{\vert\boldsymbol\gamma\vert}\\
&=\frac 1
{\vert\boldsymbol\gamma\vert}
{\binom {\vert\boldsymbol\gamma\vert-\gamma_0} { \gamma_0-1}}P\ .
\end{align*}
This shows that
\begin{equation} \label{eq:Res} 
\coef{\mathbf x^{\boldsymbol \gamma}}g_0=
\frac{\vert\boldsymbol\gamma\vert^{d-1}(\vert\boldsymbol\gamma\vert-\gamma_0)!}
{(\gamma_0-1)!\,(\vert\boldsymbol\gamma\vert+1-2\gamma_0)!}
\prod_{j=1}^{d}\frac{(\vert\boldsymbol\gamma\vert-\gamma_j-1)!}
{\gamma_j!\,(\vert\boldsymbol\gamma\vert-2\gamma_j)!}\ .
\end{equation}
Now we should remember that we actually wanted to compute the
coefficient of\break 
$x_0^{\gamma_0}x_1^{\gamma_1-1}x_2^{\gamma_2-1}\cdots x_d^{\gamma_d-1}$ in
the series $g_0(x_0,x_1,\dots,x_{d})$.
So, we have to replace $\gamma_i$ by $\gamma_i-1$ for $i=1,2,\dots,d$
and, thus,
$\vert\boldsymbol\gamma\vert$ by $s=-d+\sum_{j=0}^d\gamma_j$ in
\eqref{eq:Res}. If we do this, then we arrive easily 
at \eqref{formdimd}.
\end{proof}

\begin{prop} \label{propdetjac} 
Let $d$ be a non-negative integer and $J_{d+1}$ be the
$(d+1)\times(d+1)$ matrix
$$
\begin{pmatrix} 
\begin{cases} \frac{1+x_i}{\prod_{k=0}^{d}(1+x_k)}&i=j\\
-\frac{x_i(1+x_i)}{(1+x_j)\prod_{k=0}^{d}(1+x_k)}&i\ne j
\end{cases}
\end{pmatrix}_{0\le i,j\le d}\ .
$$
Then we have
\begin{equation} \label{eq:detJ} 
\det(J_{d+1})=
\bigg(1-\sum_{k=0}^{d}\frac {x_k} {1+2x_k}\bigg)
\prod_{j=0}^{d}\frac{1+2x_j}{(1+x_j)^{d+1}}\ .
\end{equation}
\end{prop}

\begin{proof}By factoring terms that only depend on the row index
or only on the column index, we see that
\begin{equation} \label{eq:J} 
\det(J_{d+1})=
\prod _{j=0} ^{d}\frac {1} {(1+x_j)^{d+1}}
\det\begin{pmatrix} 
\begin{cases} {1+x_i}&i=j\\
-{x_i}&i\ne j
\end{cases}
\end{pmatrix}_{0\le i,j\le d}\ .
\end{equation}
The above determinant equals the sum over all principal 
minors of the matrix
$$
\begin{pmatrix} 
\begin{cases} {x_i}&i=j\\
-{x_i}&i\ne j
\end{cases}
\end{pmatrix}_{0\le i,j\le d}\ ,
$$
where, as usual, a principal minor is by definition the determinant
of a submatrix with rows and columns indexed by a common subset of 
$\{0,1,\dots,d\}$.
Again factoring terms that only depend on the row index, we may
write the
principal minor corresponding to the submatrix 
indexed by $i_1,i_2,\dots,i_k$ in the form
\begin{equation} \label{eq:princ} 
x_{i_1}x_{i_2}\cdots x_{i_k}
\det\begin{pmatrix} 
\begin{cases} 1&i=j\\
-1&i\ne j
\end{cases}
\end{pmatrix}_{1\le i,j\le k}\ .
\end{equation}
The determinant in this expression occurs frequently.
In fact, we have
$$
\det(\lambda I_k-A_k)=\lambda^{k-1}(\lambda-k),
$$
where $I_k$ is the $k\times k$ identity matrix and 
$A_k$ the $k\times k$ all-$1$'s-matrix. 
(This is easily seen by observing that the matrix
$A_k$ has an eigenvector $(1,1,\dots,1)$ with eigenvalue $k$
and that the space orthogonal to $(1,1,\dots,1)$ is the kernel of $A_k$.)
By using this observation with $\lambda=2$, it follows that the
expression \eqref{eq:princ} simplifies to 
$$x_{i_1}x_{i_2}\cdots x_{i_k}2^{k-1}(2-k).
$$
If this is substituted in \eqref{eq:J}, we obtain
\begin{equation} \label{eq:J2} 
\det(J_{d+1})=\prod _{j=0} ^{d}\frac {1} {(1+x_j)^{d+1}}
\sum_{k=0}^{d+1}2^{k-1}(2-k)e_k(x_0,x_1,\dots,x_{d})\ ,
\end{equation}
where $e_k(x_0,x_1,\dots,x_{d})=\sum_{0\le i_1<\dots<i_k\le d}
x_{i_1}x_{i_2}\cdots x_{i_k}$ denotes the $k$-th elementary
symmetric function. As is well-known, these polynomials 
satisfy the generating function identity
\begin{equation}\label{formsym}
\sum_{k=0}^{d+1}e_k(x_0,x_1,\dots,x_{d})\,t^k=\prod_{j=0}^{d}(1+x_jt)\ .
\end{equation}
By differentiating this identity with respect to $t$, we obtain
the further equation
$$\sum_{k=1}^{d+1}k\,e_k(x_0,x_1,\dots,x_{d})\,t^{k-1}=
\sum_{k=0}^{d}\frac{x_k}{1+x_kt}\prod_{j=0}^{d}(1+x_jt)\ .$$
Using both with $t=2$ in \eqref{eq:J2}, we arrive exactly at
the right-hand side of \eqref{eq:detJ}.
\end{proof}

%%%%%%%%%%%%%%%%%%%%%%%%%%%%%%%%%%%%%%%%%%%%%%%%%%%%%%%%%%%%%%%%%%%%%%%%

\section{Proof of Theorem~\ref{thmspec}}\label{sectspec}

We perform a reverse induction on $k$.
For the start of the induction, we remember that Theorem~\ref{thmspec}
is nothing but Theorem~\ref{thmgeneralcase} (which we established in
the previous section)
if $k=d$ and $\beta_0=\beta_1=\beta_2=\dots=\beta_d=1$.

For the induction step, we have to distinguish two cases.
Suppose first that $\beta_0=1$ and that Theorem~\ref{thmspec}
holds for all (suitable) sequences 
$\beta_0=1,\beta_1,\beta_2,\dots,\beta_{k+1}$.
Then Theorem~\ref{thmspec} holds
for $\beta_0=1,\beta_1+\beta_2,\beta_3,\dots,\beta_{k+1}$ 
if and only if
\begin{multline*}
s\sum_{k=\beta_1}^{\gamma-\beta_2}\frac{\beta_1}{\beta_1s+\beta_1-k}
{\binom {\beta_1s+\beta_1-k} { k-\beta_1}}
\frac{\beta_2}{\beta_2 s+\beta_2-(\gamma-k)}
{\binom {\beta_2 s+\beta_2-(\gamma-k)} { \gamma-k-\beta_2}}\\
=\frac{(\beta_1+\beta_2)} {(\beta_1+\beta_2)(s+1)-\gamma}
{\binom {(\beta_1+\beta_2)(s+1)-\gamma} { \gamma-\beta_1-\beta_2}}
\end{multline*}
for all $\gamma\geq \beta_1+\beta_2$.
(Without loss if generality, it suffices to consider the addition
of $\beta_1$ and $\beta_2$, since all other combinations lead to
analogous and equivalent statements.)
This is a special case of an identity commonly attributed to Rothe \cite{RothAA}
(to be precise, it is the case $\alpha\to\beta_1s$, $\beta\to-1$,
$\gamma\to\beta_2s+\beta_1+\beta_2$, $n\to\gamma-\beta_1-\beta_2$ of 
\cite[Eq.~(4)]{GOulAA};
see \cite{StreAC} for historical comments and more on this kind of identities,
although, for some reason, it misses \cite{CarlAP}), which establishes
the induction step in this case.

Suppose now that Theorem~\ref{thmspec}
holds for all (suitable) sequences $\beta_0,\beta_1,\beta_2,\dots,\beta_{k+1}$.
Then Theorem~\ref{thmspec} holds
for $\beta_0+\beta_1,\beta_2,\beta_3,\dots,\beta_{k+1}$ if and only if
\begin{multline*}
\sum_{k=\beta_0}^{\gamma-\beta_1}
\frac{\beta_1s}{\beta_0 s+\beta_0-k}
{\binom {\beta_0 s+\beta_0-k} { k-\beta_0}}
{\binom {\beta_1 s+\beta_1-1-(\gamma-k)} { \gamma-k-\beta_1}}\\
={\binom {(\beta_0+\beta_1)(s+1)-1-\gamma} { \gamma-\beta_0-\beta_1}}
\end{multline*}
for all $\gamma\geq \beta_0+\beta_1$.
(Again, without loss if generality, it suffices to consider the addition
of $\beta_0$ and $\beta_1$.)
This is a special case of another identity commonly attributed to Rothe \cite{RothAA}
(to be precise, it is the case $\alpha\to\beta_0s$, $\beta\to-1$,
$\gamma\to\beta_1s+\beta_0+\beta_1-1$, $n\to\gamma-\beta_0-\beta_1$ of 
\cite[Eq.~(11)]{GOulAA}), establishing the induction step
in this case also.\qed

%For case of identified variables: see Carlitz, Some expansion associated ..
%to MacMahon's Master Theorem, SIAM J. Math. 8 (1977), 320--336. 
%(see paper 100 of CK.)
%Setting 
%$$A_{k,n}(\alpha,\beta)=\frac{bk\alpha+cn\beta+\alpha\beta}{
%(ak+cn+\alpha)(bk+dn+\beta)}{\binom {ak+cn+\alpha} { k}}
%{\binom {bk+dn+\beta} { n}}$$
%Carlitz's summation formula is the equality 
%$$\sum_{k_1,n_1\geq 0}A_{k_1,n_1}(\alpha,\beta)A_{k-k_1,n-n_1}(\alpha',\beta')
%=A_{k,n}(\alpha+\alpha',\beta+\beta')\ .$$

\end{document}